\documentclass[12pt]{article}
\usepackage{amsthm}
\usepackage{mathrsfs}

\usepackage{graphicx}
\usepackage{amsmath}
\usepackage{amsfonts}
\usepackage{amssymb}

\usepackage{listings}
\usepackage{listings,xcolor}

\usepackage{hyperref}

\usepackage{graphics, setspace}

\newcommand{\mathsym}[1]{{}}
\newcommand{\unicode}[1]{{}}

\usepackage{verbatim}
\usepackage[mathlines]{lineno}
\usepackage{color}
\usepackage{enumerate}
\usepackage{enumitem}
\usepackage{graphicx}
\definecolor{red}{rgb}{1,0,0}

\date{}

\newtheorem{thm}{Theorem}[section]

\newtheorem{quest}[thm]{Question}

\newtheorem{cor}[thm]{Corollary}

\newtheorem{obs}[thm]{Observation}

\def\noi{\noindent}

\begin{document}

\title{Spectral Bounds for the Connectivity of Regular Graphs with Given Order}
\author{Aida Abiad\thanks{Department of Quantitative Economics, Maastricht University,
Maastricht, The Netherlands; Department of Pure Mathematics and Computer Algebra, Ghent University, Ghent, Belgium (A.AbiadMonge@maastrichtuniversity.nl).}
\and
Boris Brimkov\thanks{Department of Computational and Applied Mathematics, Rice University,
Houston, TX 77005, USA (boris.brimkov@rice.edu).}
\and
Xavier Mart\'inez-Rivera\thanks{Department of Mathematics, Iowa State University, Ames,
IA 50011, USA (xaviermr@iastate.edu).}
\and
Suil O \thanks{Applied Mathematics and Statistics, The State University of New York Korea, Incheon, 21985, Republic of Korea (suil.o@sunykorea.ac.kr).}
\and
Jingmei Zhang\thanks{Department of
Mathematics, University of Central Florida, Orlando,
FL 32816, USA (jmzhang@knights.ucf.edu).}
}


\maketitle

\begin{abstract}
The second-largest eigenvalue and second-smallest Laplacian eigenvalue of a graph are measures of its connectivity. These eigenvalues can be used to analyze the robustness, resilience, and synchronizability of networks, and are related to connectivity attributes such as the vertex- and edge-connectivity, isoperimetric number, and characteristic path length. In this paper, we present two upper bounds for the second-largest eigenvalues of regular graphs and multigraphs of a given order which guarantee a desired vertex- or edge-connectivity. The given bounds are in terms of the order and degree of the graphs, and hold with equality for infinite families of graphs. These results answer a question of Mohar.
\end{abstract}

\noi{\bf Keywords.}
Second-largest eigenvalue;
vertex-connectivity;
edge-connectivity;
regular multigraph;
algebraic connectivity.\\

\noi{\bf AMS subject classifications.}
05C50, 05C40.

\section{Introduction}
\label{section_intro}

Determining the connectivity of a graph is a problem that arises often in various applications -- see for example \cite{kao} and \cite{resilience}. Let $\kappa(G)$ and $\kappa'(G)$ denote the vertex- and edge-connectivity of a connected graph $G$. Let $L(G)=D(G)-A(G)$ be the Laplacian matrix of $G$, where $D(G)$ is the diagonal degree matrix of $G$ and $A(G)$ is the adjacency matrix of $G$. We denote the eigenvalues of $A(G)$ by $\lambda_1(G)\geq \cdots\geq \lambda_n(G)$ and the eigenvalues of $L(G)$ by $0=\mu_1(G)\leq\cdots\leq \mu_n(G)$. In 1973, Fiedler related the vertex-connectivity of a graph $G$ to $\mu_2(G)$ as follows:

\begin{thm}[{\rm Fiedler \cite{Fiedler}}]
\label{thm1.1}
If $G$ is a simple, non-complete graph, then $\kappa(G) \geq \mu_2(G)$.
\end{thm}

\noindent This seminal result provided researchers with another parameter that quantitatively measures the connectivity of a graph; hence, $\mu_2(G)$ is known as the \emph{algebraic connectivity} of $G$. Fiedler's discovery ignited interest in studying the connectivity of graphs by analyzing the spectral properties of their associated matrices.
Akin to other connectivity measures such as vertex-connectivity, edge-connectivity, and isoperimetric number, the algebraic connectivity of a graph has
applications in the design of reliable communication networks \cite{liu} and in analyzing the robustness of complex networks \cite{jamakovic1,jamakovic2}.

Recall that for a $d$-regular multigraph $G$ on $n$ vertices, $\lambda_i(G) = d - \mu_i(G)$ for $i = 1,\dots, n$. Thus, for regular multigraphs, spectral
bounds related to connectivity are often expressed in terms of the second-largest eigenvalue, instead of the second-smallest Laplacian eigenvalue.

\subsubsection*{Literature review}

Below we survey several results relating $\lambda_2(G)$ to $\kappa'(G)$. Note that Theorem~\ref{thm1.1} implies $\kappa'(G)\geq \mu_2(G)$, since $\kappa'(G)\geq \kappa(G)$.

\begin{thm}[Chandran {\rm \cite{Chandran 2004}}]
Let $G$ be an $n$-vertex $d$-regular simple graph with $\lambda_2(G) < d-1 - \frac{d}{n-d}$. Then $\kappa'(G) = d$.
\end{thm}

\begin{thm}[Krivelevich and Sudakov {\rm \cite{K & Sudakov book chapter}}]
\label{KSbound}
Let $G$ be a $d$-regular simple graph with $\lambda_2(G) \leq d-2$. Then $\kappa'(G) \geq d$.
\end{thm}

\noindent In 2010, Theorem \ref{KSbound} was improved by Cioab\u{a} \cite{Cioaba 2010} as follows.
\begin{thm}[Cioab\u{a} {\rm \cite{Cioaba 2010}}]\label{cioba}
Let $t$ be a nonnegative integer less than $d$, and let $G$ be a $d$-regular, simple graph with $\lambda_2(G) < d - \frac{2t}{d+1}$.
Then $\kappa'(G) \geq t+1$.
\end{thm}

\noindent In the same paper, Cioab\u{a} also gave improvements of Theorem \ref{cioba} for the following two particular cases.

\begin{thm}[Cioab\u{a} {\rm \cite{Cioaba 2010}}]\label{Cioba improvement for t=1}
Let $d \geq 3$ be an odd integer and let $\pi(d)$
denote the largest root of
$x^3 - (d-3)x^2 - (3d-2)x - 2 = 0$.
If $G$ is a $d$-regular, simple graph such that
$\lambda_2(G) <\pi(d)$, then $\kappa'(G) \geq 2$.
\end{thm}

\noindent The value of $\pi(d)$ above is approximately $d - \frac{2}{d+5}$.

\begin{thm}[Cioab\u{a} {\rm \cite{Cioaba 2010}}]\label{Cioba improvement for t=2}
Let $d \geq 3$ be any integer.
Let $G$ be a $d$-regular, simple graph with
\begin{equation*}
\lambda_2(G) < \frac{d-3+ \sqrt{(d+3)^2 - 16}}{2}.
\end{equation*}
Then $\kappa'(G) \geq 3$.
\end{thm}

\noindent The value of $\frac{d-3+ \sqrt{(d+3)^2 - 16}}{2}$ above is approximately $d - \frac{4}{d+3}$. Note that Theorems \ref{Cioba improvement for t=1} and \ref{Cioba improvement for t=2} are best possible, as there are examples showing that the upper bounds cannot be lowered. The following extension of these results to $t\geq 3$ was conjectured in the Ph.D. thesis of the fourth author~\cite{O-thesis} and was resolved in \cite{suil_in_prep}. 

\begin{thm}[O, Park, Park, and Yu {\rm \cite{suil_in_prep}}]
Let $3 \le t \le d-1$ and let $G$ be a $d$-regular simple graph with 
\begin{equation*}
\lambda_2(G)<
\begin{cases}
\frac{d-3+\sqrt{(d+3)^2-8t}}{2} &\text{if t is even}\\
\frac{d-4+\sqrt{(d+4)^2-8t}}{2} &\text{if t is odd.}
\end{cases}
\end{equation*}
Then $\kappa'(G) \ge t+1$.

\end{thm}

In 2016, O \cite{O-Edge-conn from eigenvalues} generalized Fiedler's result to multigraphs, and established similar bounds to those above.

\begin{thm}[O {\rm \cite{O-Edge-conn from eigenvalues}}]\label{O, bound 1 for multigraphs}
Let $G$ be a connected, $d$-regular multigraph with
\begin{equation*}
\lambda_2(G) < \frac{d-1+ \sqrt{9d^2-10d+17}}{4}.
\end{equation*}
Then $\kappa'(G) \geq 2$.
\end{thm}

\begin{thm}[O {\rm \cite{O-Edge-conn from eigenvalues}}]\label{O, bound 2 for multigraphs}
Let $t \geq 2$ and let $G$ be a connected,
$d$-regular multigraph.
If $\lambda_2(G) < d-t$, then $\kappa'(G) \geq t+1$.
If $t$ is odd and $\lambda_2(G) < d-t+1$,
then $\kappa'(G) \geq t+1$.
\end{thm}

Note that Theorem \ref{O, bound 2 for multigraphs} is best possible for multigraphs. For every $0<t<d$, O \cite{O-Edge-conn from eigenvalues} found examples where the bound in Theorem \ref{O, bound 2 for multigraphs} is tight.


The results above make assertions about the edge-connectivity of a graph based on its eigenvalues. In  more recent papers, Cioab\u{a} and Gu \cite{CioabaGu2016} and O \cite{O- alg conn mult graphs}
also established analogous results for vertex-connectivity.

\begin{thm}[Cioab\u{a} and Gu {\rm \cite{CioabaGu2016}}]\label{CioabaGu16}
Let $G$ be a connected $d$-regular simple graph, $d\geq 3$, and
\begin{equation*}
\lambda_2(G)<
\begin{cases}
\frac{d-2+\sqrt{d^2+12}}{2} &\text{if d is even}\\
\frac{d-2+\sqrt{d^2+8}}{2} &\text{if d is odd.}
\end{cases}
\end{equation*}
Then, $\kappa(G)\geq 2$.
\end{thm}

\begin{thm}[O {\rm \cite{O- alg conn mult graphs}}]\label{Suilthm1.9}
Let $G$ be a $d$-regular multigraph that is not the $2$-vertex $d$-regular multigraph.
If $\lambda_2(G) < \frac{3d}{4}$, then $\kappa(G) \geq 2$.
\end{thm}

See \cite{abreu,kirkland,rad} and the bibliographies therein for other recent results on algebraic connectivity; see also \cite{dam1,dam2,dam3} for characterizations of the algebraic connectivities of specific families of graphs.


\subsubsection*{Main contributions}

The aim of the present paper is to investigate what upper bounds on the second-largest eigenvalues of regular simple graphs and multigraphs of a given order guarantee a
desired vertex-connectivity $\kappa(G)$ or edge-connectivity $\kappa'(G)$. In other words, we address the following question asked by Mohar (private communication with the fourth author) and alluded to in \cite{CioabaGu2016}:

\begin{quest}
\label{q1.11}
For a $d$-regular simple graph or multigraph $G$ of a given order and for $1 \leq t \leq d-1$,
what is the best upper bound for $\lambda_2(G)$ which guarantees that
$\kappa'(G) \geq t+1$ or that $\kappa(G) \geq t+1$?
\end{quest}

A starting point of our work, which also  motivated the above question, comes from Theorem \ref{O, bound 2 for multigraphs} \cite{O-Edge-conn from eigenvalues}, because despite the fact that the bound was shown to be tight, the tightness comes from the smallest multigraph. This suggests that this bound can be improved, and a natural next step is to look at the case where the number of vertices is fixed. The main results of this work are the following two spectral bounds which guarantee a certain vertex- and edge-connectivity for multigraphs of a given order. We also construct examples which show the bounds are tight.

\begin{thm}
\label{thm112}
Let $G$ be an $n$-vertex $d$-regular multigraph with $n\geq 5$ and $d\geq 3$. If $\lambda_2(G)<\frac{8n-25}{9n-25}d$, then $\kappa(G)\geq 2$.
\end{thm}

\begin{thm}
\label{thm113}
Let $G$ be an $n$-vertex $d$-regular multigraph with $\lambda_2(G)<\rho(d,n)$,
where $\rho(d,n)$ is the second-largest eigenvalue of a certain $4\times 4$ matrix (see Section \ref{section4}). Then $\kappa'(G)\geq 2$.
\end{thm}

Theorems \ref{thm112} and \ref{thm113} extend the results listed earlier to multigraphs, and improve some of them (e.g. Theorem \ref{Suilthm1.9}). The majority of the related results listed earlier were derived using a variety of combinatorial, linear algebraic, and analytic techniques; moreover, they
feature upper bounds for $\lambda_2(G)$ which do not depend on the order of the graph. In contrast, the results derived in the present paper feature bounds
for $\lambda_2(G)$ which depend on both the degree and the order of the graphs, and as such are tight for infinite families of graphs. Furthermore, the
derivations of these results combine analytic techniques with computer-aided \emph{symbolic} algebra; this proves to be a powerful approach, easily
establishing the desired results in all but finitely-many cases. The remaining cases are verified through a brute-force approach which relies on
enumerating all multigraphs with certain properties. In order to avoid enumeration and post-hoc elimination of the exponential number of multigraphs
without the desired properties, our approach required the development of novel combinatorial and graph theoretic techniques. While the problem of
generating all non-isomorphic simple graphs having a certain degree sequence and other properties is well-studied (cf. \cite{hakimi,hanlon,ruskey}), there
are not as many efficiently-implemented algorithms for constrained enumeration of multigraphs (see \cite{taqqu} for some results in this direction). Thus,
the developed enumeration procedure may also be of independent interest.

The paper is organized as follows. In the next section, we recall some graph theoretic and linear algebraic notions, specifically those related to eigenvalue interlacing. In Sections 3 and 4, we present our main results. We conclude with some final remarks in Section 5. The Appendix includes further details and computer code for symbolic computations used in some of the proofs.

We note that Theorems \ref{theorem1} and \ref{theorem4} are not our main results and are not tight, but we include them for completeness since they are general bounds that give a better intuition of the bigger picture. Also, note that the results for simple graphs discussed in this section are not comparable with our bounds for multigraphs in Theorems \ref{thm112} and \ref{thm113}. See for instance the upper bound on $\lambda_2$ in Theorem \ref{Cioba improvement for t=2} \cite{Cioaba 2010}, which for $t=d-1$ behaves approximately as $d$, and the upper bound $\lambda_2$ in Theorem \ref{O, bound 2 for multigraphs}
\cite{O-Edge-conn from eigenvalues}, which for $t=d-1$ behaves as a small constant. Hence, there is a large gap between the upper bounds on the second largest eigenvalue in simple graphs and the upper bounds for multigraphs, which suggests that there may well be room for improvement.

\section{Preliminaries}
\label{preliminaries}
In this paper, a \emph{multigraph} refers to a graph with multiple edges but no loops; a \emph{simple graph} refers to a graph with no multiple edges or
loops. The \emph{order} and \emph{size} of a multigraph $G$ are denoted by $n=|V(G)|$ and $m=|E(G)|$, respectively. A \emph{double edge} (respectively
\emph{triple edge}) in a multigraph is an edge of multiplicity two (respectively three). The \emph{degree} of a vertex $v$ of $G$, denoted $d_G(v)$, is the
number of edges incident to $v$. The \emph{degree sequence} of $G$ is a list $\{d_1,\ldots,d_n\}$ of the vertex degrees of $G$. We may abbreviate the
degree sequence of $G$ by only writing distinct degrees, with the number of vertices realizing each degree in superscript. For example, if $G$ is the star
graph on $n$ vertices, the degree sequence of $G$ may be written as $\{n-1, 1^{n-1}\}$.

A \emph{vertex cut} (respectively \emph{edge cut}) of $G$ is a set of vertices (respectively edges) which, when removed, increases the number of connected
components in $G$. A multigraph $G$ with more than $k$ vertices is said to be \emph{$k$-vertex-connected} if there is no vertex cut of size $k-1$. The
\emph{vertex-connectivity} of $G$, denoted $\kappa(G)$, is the maximum $k$ such that $G$ is $k$-vertex-connected. Similarly, $G$ is
\emph{$k$-edge-connected} if there is no edge cut of size $k-1$; the \emph{edge-connectivity} of $G$, denoted $\kappa'(G)$, is the maximum $k$ such that
$G$ is $k$-edge-connected. A \emph{cut-vertex} (respectively \emph{cut-edge}) is a vertex cut (respectively edge cut) of size one.

Given sets $V_1,V_2\subset V(G)$, $[V_1,V_2]$ denotes the number of edges with one endpoint in $V_1$ and the other in $V_2$. The \emph{induced subgraph}
$G[V_1]$ is the subgraph of $G$ whose vertex set is $V_1$ and whose edge set consists of all edges of $G$ which have both endpoints in $V_1$. A
\emph{matching} is a set of edges of $G$ which have no common endpoints; a $k$-matching is a matching containing $k$ edges. $G+e$ denotes the graph
$(V(G),E(G)\cup\{e\})$, and $G+E'$ denotes the graph $(V(G),E(G)\cup E')$. The \emph{complete graph} on $n$ vertices is denoted $K_n$. An \emph{odd path}
(respectively \emph{even path}) in a graph is a connected component which is a path with an odd (respectively even) number of vertices. For other graph
theoretic terminology and definitions, we refer the reader to \cite{west}.

The \emph{adjacency matrix} of $G$ will be denoted by $A(G)$; recall that in a multigraph, the entry $A_{i,j}$ is the number of edges between vertices
$v_i$ and $v_j$. The \emph{eigenvalues} of $G$ are the eigenvalues of its adjacency matrix, and are denoted by $\lambda_1(G) \geq \cdots \geq
\lambda_n(G)$. The \emph{Laplacian matrix} of $G$ is equal to $D(G) - A(G)$, where $D(G)$ is the diagonal matrix whose entry $D_{i,i}$ is the degree of
vertex $v_i$. The \emph{Laplacian eigenvalues} of $G$ are the eigenvalues of its Laplacian matrix and are denoted by $0=\mu_1(G) \leq \cdots \leq
\mu_n(G)$. The dependence of these parameters on $G$ may be omitted when it is clear from the context. Let $A$ be an $n \times n$ matrix; $B$ is a
\emph{principal submatrix} of $A$ if $B$ is a square matrix obtained by removing certain rows and columns of $A$.

A technical tool used in this paper is \emph{eigenvalue interlacing} (for more details see Section 2.5 of \cite{BH}). Given two sequences of real numbers
$a_1\geq\cdots\geq a_n$ and $b_1\geq \cdots\geq b_m$ with $m<n$, we say that the second sequence \emph{interlaces} the first sequence whenever $a_i\geq
b_i\geq a_{n-m+i}$ for $i=1,\ldots, m$.

\begin{thm} \textup{[Interlacing Theorem,\cite{BH}]}
If $A$ is a real symmetric $n \times n$ matrix and $B$ is a principal submatrix of $A$ of order $m \times m$ with $m<n$, then for $1 \le i \le m$,
$\lambda_i(A) \ge \lambda_i(B) \ge \lambda_{n-m+i}(A)$, i.e., the eigenvalues of $B$ interlace the eigenvalues of $A$.
\end{thm}

Let $\mathcal{P}=\{V_1,\ldots,V_s\}$ be a partition of the vertex set of a multigraph $G$ into $s$ non-empty subsets. The \emph{quotient matrix} $Q$
corresponding to $\mathcal{P}$ is the $s\times s$ matrix whose entry $Q_{i,j}$ ($1\leq i,j\leq s$) is the average number of incident edges in $V_j$ of the
vertices in $V_i$. More precisely, $Q_{i,j}=\frac{[V_i,V_j]}{|V_i|}$ if $i \neq j$, and $Q_{i,i}=\frac{2|E(G[V_i])|}{|V_i|}$. Note that for a simple graph,
$Q_{i,j}$ is just the average number of neighbors between vertices in $V_j$ and vertices in $V_i$.

\begin{cor} \textup{[Corollary 2.5.4, \cite{BH}]}\label{cor}
The eigenvalues of any quotient matrix $Q$ interlace the eigenvalues of $G$.
\end{cor}

\section{Bounds for $\lambda_2(G)$ to guarantee $\kappa(G)\ge t+1$}

\subsection{$\lambda_2(G)$ and $\kappa(G)\geq t+1$}

In this section, we establish an upper bound for the second-largest eigenvalue of an $n$-vertex $d$-regular simple graph or multigraph which guarantees a
certain vertex-connectivity. To our knowledge, this is the first spectral bound on the vertex-connectivity of a regular graph which depends on both the
degree and the order of the graph.

\begin{thm}
\label{theorem1}
Let $G$ be an $n$-vertex $d$-regular simple graph or multigraph, which is not obtained by duplicating edges in a complete graph on at most $t+1$ vertices;
let

\begin{equation*}\phi(d,t)= \begin{cases}
~~2~~&\text{ if } G \text{ is a multigraph and } t=1 \\
~~1~~&\text { if } G \text { is a multigraph and } t\ge 2 \\
d+1~~&\text{ if } G \text { is a simple graph and } t=1\\
d+1-t&\text{ if } G \text { is a simple graph and } t \ge 2,
\end{cases}
\end{equation*}
where $0 \le t \le d-1$. If $\lambda_2(G)< d-\frac {td}{2\phi(d,t)} - \frac{td}{2(n-\phi(d,t))}$, then $\kappa(G) \ge t+1$.
\end{thm}

\proof
Assume to the contrary that $\kappa(G)\leq t$. If $G$ is disconnected, then $\lambda_2(G)=d \ge d-\frac {td}{2\phi(d,t)} - \frac{td}{2(n-\phi(d,t))}$, a
contradiction.
Now, assume that $\kappa(G) \ge 1$.
Hence, there exists a vertex cut $C$ of $G$ with $1\leq c:=|C|\leq t$. Let $S_1$ be a union of some components of $G-C$ such that
$[S,\overline{S}]=[C,\overline{S}] \le \frac{cd}{2} \le \frac{td}{2}$, where $S=S_1 \cup C$ and $\bar{S}=V(G)\backslash S$. See Figure \ref{fig_ptd} for an
illustration of this partition.

\begin{figure}[ht!]
\begin{center}
\includegraphics[scale=0.35]{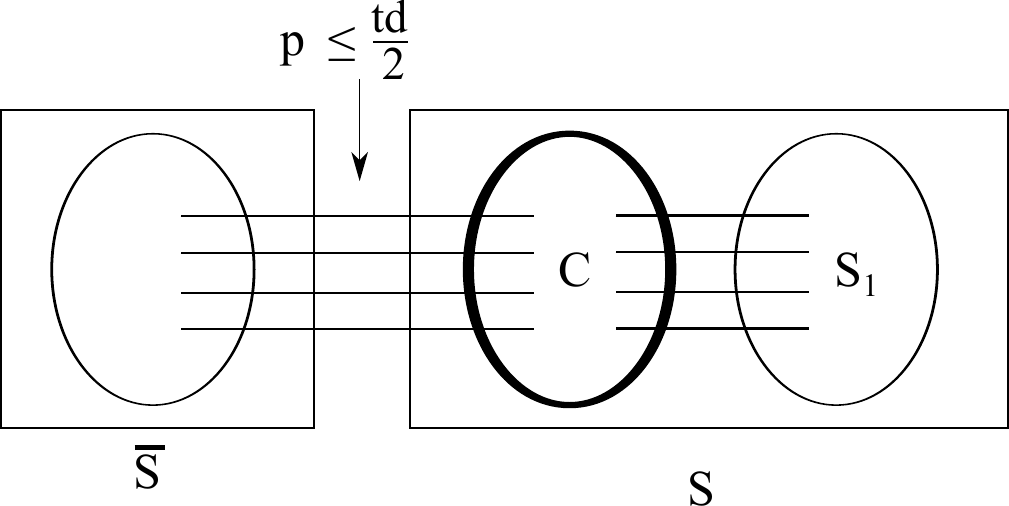}\qquad\qquad
\caption{Partition of $V(G)$ into $S$ and $\bar{S}$.}
\label{fig_ptd}
\end{center}
\end{figure}

\noindent Let $[S, \overline{S}]=p$, and $|S_1|=s_1$; then, we have $2[S,S]=d(s_1+c)-p$, and $2[\overline{S},\overline{S}]=d(n-s_1-c)-p$, so the quotient
matrix for the partition $\{S, \overline{S}\}$ is
\begin{equation*}
Q=\left(
\begin{array}{cc}
d-\frac{p}{s_1+c} & \frac{p}{s_1+c}\\
\frac{p}{n-s_1-c}& d-\frac{p}{n-s_1-c} \\
\end{array}
\right),
\end{equation*}
\noindent and the characteristic polynomial of $Q$ with respect to $x$ is $(x-d)(x-d+\frac{p}{s_1+c}+\frac{p}{n-s_1-c})$. Then
by Corollary~\ref{cor},
we have
\[\lambda_2(G) \ge d-\frac{p}{s_1+c}-\frac{p}{n-s_1-c}.\]

We now consider two cases based on whether $G$ is a simple graph or a multigraph.

\begin{description}
  {\it \item[Case 1:] $G$ is a simple graph}. If $t=1$, then $c=1$, and since the degree of each vertex in $S_1$ is $d$, it holds that $s_1 \ge d$. If $s_1 = d$, $G[S]$
  is a complete subgraph of $G$, so the vertex in $C$ has degree greater than $d$ because $p \ge 1$; this is a contradiction. Thus $s_1 \ge d+1$ and $p \le \frac{d}{2}$. Moreover, since $n\geq s_1+d+2$, it follows that $\frac{d+1}{s_1+1}\leq \frac{n-(s_1-1)}{n-(d+1)}$, and hence $\frac{1}{(s_1+1)(n-(s_1+1))}\leq \frac{1}{(d+1)(n-(d+1))}$. Using this inequality, we have
\begin{eqnarray*}
\lambda_2(G) &\geq& d-\frac{p}{s_1+c}-\frac{p}{n-s_1-c}\geq d-\frac{d}{2(s_1+1)}-\frac{d}{2(n-(s_1+1))}\\
&=&d-\frac{dn}{2}\frac{1}{(s_1+1)(n-(s_1+1))}\geq d-\frac{dn}{2}\frac{1}{(d+1)(n-(d+1))}\\
&=&d-\frac{d}{2(d+1)}-\frac{d}{2(n-d-1)},
\end{eqnarray*}
as desired. If $t \ge 2$, by the same argument as above, we must have
  $s_1 \ge d+1-c\geq d+1-t$, $n\geq d+1+s_1$, $p \le \frac{td}{2}$, and so $\lambda_2(G) \ge d-\frac{td}{2(d+1-t)}-\frac{td}{2(n-d-1+t)}$, as desired.

  {\it \item[Case 2:] $G$ is a multigraph}. If $t=1$, then $c=1$, $s_1 \ge 2$, $p \le d/2$. Moreover, since $n\geq s_1+3$, it follows that 
$\frac{2}{s_1+1}\leq \frac{n-(s_1+1)}{n-2}$ and hence $\frac{1}{(s_1+1)(n-(s_1+1))}\leq \frac{1}{2(n-2)}$. Using this inequality, we have
\begin{eqnarray*}
\lambda_2(G) &\geq& d-\frac{p}{s_1+c}-\frac{p}{n-s_1-c}\geq d-\frac{d}{2(s_1+1)}-\frac{d}{2(n-(s_1+1))}\\
&=&d-\frac{dn}{2}\frac{1}{(s_1+1)(n-(s_1+1))}\geq d-\frac{dn}{2}\frac{1}{2(n-2)}\\
&=&d-\frac{d}{4}-\frac{d}{2(n-2)},
\end{eqnarray*}
as desired. If $t \ge 2$, then $s_1 \ge 1$, $p \leq \frac{td}{2}$, $n\geq s_1+c+1$, and by a similar reasoning as above, $\lambda_2(G) \geq d- \frac{td}{2}-\frac{td}{2(n-1)}$, as desired. 
  \qed 
 
\end{description}

\subsection{Improved bound for $\lambda_2(G)$ to guarantee $\kappa(G)\geq 2$}

We now improve the result of Theorem \ref{theorem1} for the case when $G$ is a multigraph and $t=1$. Recall that in this case, Theorem \ref{theorem1} states
that if $\lambda_2(G)<d-\frac{d}{4}-\frac{d}{2(n-2)}=\frac{3n-8}{4n-8}d$, then $\kappa(G)\geq 2$. Moreover, in Observation \ref{theo32tight} it is shown that the following bound from Theorem \ref{theo23} is tight. As discussed in Section \ref{section_intro}, the bound of Theorem \ref{theo23} is incomparable with bounds on $\lambda_2(G)$ guaranteeing a certain vertex connectivity for simple graphs (e.g. Theorem~\ref{CioabaGu16}); however, it does improve the bound of Theorem \ref{Suilthm1.9} for multigraphs.

\begin{thm}\label{theo23}
Let $G$ be an $n$-vertex $d$-regular multigraph with $n\geq 5$ and $d\geq 3$. If $\lambda_2(G)<\frac{8n-25}{9n-25}d$, then $\kappa(G)\geq 2$.
\end{thm}

\begin{proof}
Assume to the contrary that $\kappa(G) \leq 1$.
If $\kappa(G)=0$, then $\lambda_2(G)=d >\frac{8n-25}{9n-25}d$, a contradiction. Thus, we can assume henceforth that $\kappa(G)=1$.

Let $v$ be a cut-vertex of $G$, and $S_1$ and $S_2$ be two components of $G-v$ with $|S_1|=s_1$ and $|S_2|=s_2=n-s_1-1$. Let $m_1=[v,S_1]$ and
$m_2=[v,S_2]$; without loss of generality, we can assume that $m_2\leq m_1$, and hence that $1\leq m_2\leq \frac{d}{2}$ (otherwise the roles of $S_1$ and
$S_2$ can be reversed); note that since $d\geq 3$, we must have $2\leq s_1\leq n-3$; moreover, $d=m_1+m_2$. See Figure \ref{fig2} for an illustration of
this partition in the case when $s_1=2$.

\begin{figure}[ht!]
\begin{center}
\includegraphics[scale=0.35]{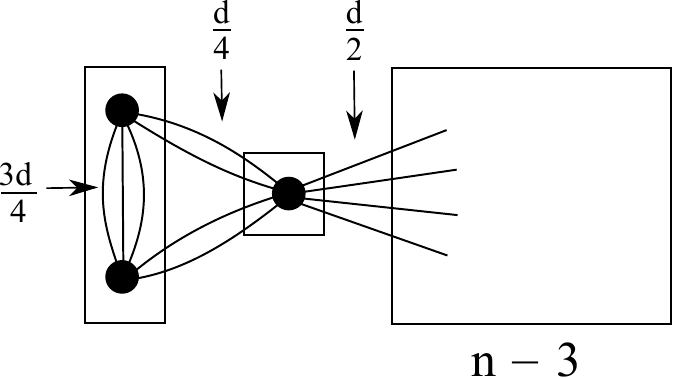}\qquad\qquad
\caption{Partition of $V(G)$ into $S_1$, $\{v\}$ and $S_2$, when $|S_1|=2$.}
\label{fig2}
\end{center}
\end{figure}

\noindent The quotient matrix for the partition $\{S_1, \{v\}, S_2\}$ is

\begin{equation*}
Q=\left( \begin{array}{ccc}
d- \frac{m_1}{s_1} & \frac{m_1}{s_1} & 0 \\
m_1 & 0 & m_2 \\
0 & \frac{m_2}{s_2} & d- \frac{m_2}{s_2}
\end{array} \right),
\end{equation*}
and its characteristic polynomial with respect to $x$ is
\begin{equation*}
(x-d)\left[x^2-
\left(d-\frac{m_1}{s_1}-\frac{m_2}{s_2}\right)x - \frac{m_1^2}{s_1} -
\frac{m_2^2}{s_2} + \frac{m_1 m_2}{s_1 s_2}\right].
\end{equation*}
Then by Corollary~\ref{cor}, we have $\lambda_2(G) \geq \lambda_2(Q)$, where $\lambda_2(Q)$ is the second-largest root of the characteristic polynomial of
$Q$; it can be verified that $\lambda_2(Q)$ can be expressed as follows:
\begin{equation}
\label{eq1}
\frac{1}{2}\left[
d-\frac{m_1}{s_1}-\frac{m_2}{s_2}+
\sqrt{\left(d-\frac{m_1}{s_1}-\frac{m_2}{s_2}\right)^2 +
4\left(\frac{m_1^2}{s_1} +
\frac{m_2^2}{s_2} - \frac{m_1 m_2}{s_1 s_2}\right)}\right].
\end{equation}
If we set the derivative of $\lambda_2(Q)$ with respect to $m_2$ equal to zero and solve for $m_2$, we obtain
\begin{equation}
\label{eq2}
m_2=\frac{d(s_2+2s_1s_2)}{n-1+4s_1s_2}.
\end{equation}
Substituting $d-m_2$ for $m_1$, and the right hand side of (\ref{eq2}) for $m_2$ in (\ref{eq1}), and simplifying, we obtain
\begin{equation*}
\lambda_2(G)\geq d-\frac{dn}{n-1+4s_1s_2}.
\end{equation*}
Finally, when we substitute $n-s_1-1$ for $s_2$, the resulting expression has a minimum at $s_1=2$, for $n\geq 5$, $d\geq 3$, and $2\leq s_1\leq n-3$, with
minimal value $\frac{8dn-25d}{9n-25}$. This minimization and some of the algebraic manipulations described above were carried out using symbolic
computation in Mathematica; for details, see the Appendix.
\end{proof}

\begin{obs}
\label{theo32tight}
Let $G$ be a multigraph with the following adjacency matrix:

\begin{equation*}
\left( \begin{array}{ccccc}
0&3d/4&d/4&0&0\\
3d/4&0&d/4&0&0\\
d/4&d/4&0&d/4&d/4\\
0&0&d/4&0&3d/4\\
0&0&d/4&3d/4&0\end{array} \right).
\end{equation*}
Then $\lambda_2(G)=\frac{8\cdot 5-25}{9\cdot 5-25}d$. Moreover, $G$ is a $d$-regular multigraph with 5 vertices, $d=4k$, $k\geq 1$, and $\kappa(G)=1$.
Thus, the bound in Theorem \ref{theo23} is the best possible for this infinite family of multigraphs.
\end{obs}

\section{Bounds for $\lambda_2(G)$ to guarantee $\kappa'(G)\ge t+1$}
\label{section4}

In this section, we first give an upper bound for $\lambda_2(G)$ in an $n$-vertex $d$-regular multigraph which guarantees that $\kappa'(G) \ge t+1$; its
proof is omitted, since it is similar to that of Theorem \ref{theorem1}. Theorem \ref{theorem4} extends a result of  Cioab\u{a}~\cite{Cioaba 2010} to
multigraphs.

\begin{thm}
\label{theorem4}
Let $G$ be an $n$-vertex $d$-regular multigraph, which is not obtained by duplicating edges in a complete graph on at most $t+1$ vertices. Let

\begin{equation*}\psi(d,t)= \begin{cases}
3\quad\text{ if } t=1 \\
2\quad\text { if } t\ge 2, \\
\end{cases}
\end{equation*}
where $0 \le t \le d-1$. If $\lambda_2(G)< d-\frac {t}{\psi(d,t)} - \frac{t}{n-\psi(d,t)}$, then $\kappa'(G) \ge t+1$.
\end{thm}

\noindent Now, we will improve the bound in Theorem~\ref{theorem4} for the case of $t=1$; see Observation \ref{obs_improvement} for an explanation of why
Theorem \ref{theo31} is an improvement. In Observation \ref{theo4.2tight}, it is shown that the bound in Theorem~\ref{theo31} is tight.

\begin{thm}\label{theo31}
Let $G$ be an $n$-vertex $d$-regular multigraph with $\lambda_2(G)<\rho(d,n)$,
where $\rho(d,n)$ is the second-largest eigenvalue of the following matrix:
\begin{equation*}\label{Qalsoforthm41}
Q=\left( \begin{array}{cccc}
\frac{d+1}{2} & \frac{d-1}{2} & 0 & 0 \\
d-1 & 0 & 1 & 0 \\
0 & 1 & 0 & d-1 \\
0 & 0 & \frac{d-1}{n-4} & d-\frac{d-1}{n-4} \end{array} \right).
\end{equation*}
Then $\kappa'(G)\geq 2$.
\end{thm}

\proof
Assume to the contrary that $\kappa'(G)\leq 1$. If $\kappa'(G)=0$, then since the largest eigenvalue of $Q$ equals $d$, we have that
$\lambda_2(G)=d=\lambda_1(Q)\geq \lambda_2(Q)=\rho(d,n)$, a contradiction.

Now, assume that $\kappa'(G)=1$.
For any graph $H$, define $sc(H)$ to be the number of vertices in the smallest connected component of $H$.
Let $e=v_1v_2$ be a cut-edge of $G$ such that $sc(G-e)=\min\{sc(G-f):f \text{ is a cut-edge of }G\}$. In other words, $e$ is a cut-edge such that one of
the components of $G-e$ has minimum size among all subgraphs of $G$ which can be separated by removing a cut-edge of $G$. Let $G_1$ and $G_2$ be the two
components of $G-e$, where $v_1\in G_1$, $v_2\in G_2$, and $|V(G_1)|\leq |V(G_2)|$. For $i\in \{1,2\}$, let $S_i=V(G_i)\backslash\{v_i\}$ and $s_i=|S_i|$.
By the degree-sum formula, $ds_i+(d-1)=\sum_{v\in V(G_i)}d_{G_i}(v)=2|E(G_i)|$, whence it follows that $d(s_i+1)$ is odd. Thus, both $d$ and $s_i+1$ are
odd, and hence $n$ is even; moreover, $s_i\ge 2$, and hence $n\geq 6$. See Figure~\ref{partition} for an illustration.

\begin{figure}[ht!]
\begin{center}
\includegraphics[scale=0.35]{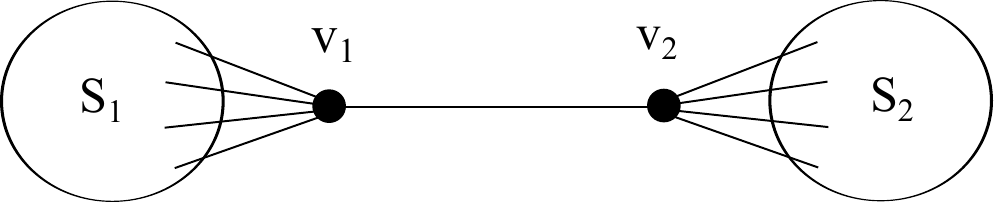}\qquad\qquad
\caption{A $d$-regular multigraph with $\kappa'(G)=1$.}
\label{partition}
\end{center}
\end{figure}

\noindent We now consider three cases based on the cardinality of $s_1+1$.

\begin{description}

  \item[Case 1: $s_1+1=3$.]

In this case, the structure of the graph is determined uniquely, and the vertex partition $\{S_1, \{v_1\},\{v_2\}, S_2\}$ corresponds to the quotient
matrix $Q$ defined in the statement of the Theorem; see Figure \ref{case_n3} for an illustration. Therefore, the inequality $\lambda_2(G)\geq\rho(d,n)$
holds for all $d$ and $n$.

\begin{figure}[ht!]
\label{case_n3}
\begin{center}
\includegraphics[scale=0.35]{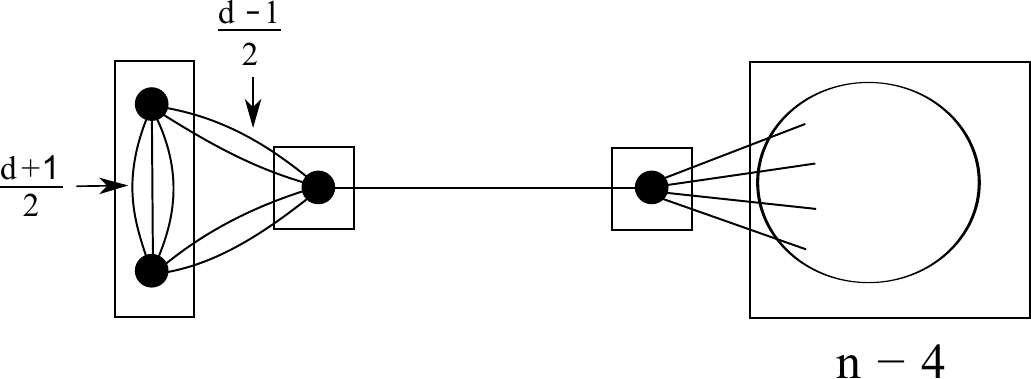}\qquad\qquad
\caption{A $d$-regular multigraph with $\kappa'(G)=1$ and $s_1=2$.}
\end{center}
\end{figure}

\item[Case 2: $s_1+1=5$.]

Consider the partition $\{S_1, \{v_1\},\{v_2\}, S_2\}$ and the corresponding quotient matrix:

\begin{equation*}
Q'=\left( \begin{array}{cccc}
d-\frac{d-1}{4} & \frac{d-1}{4} & 0 & 0 \\
d-1 & 0 & 1 & 0 \\
0 & 1 & 0 & d-1 \\
0 & 0 & \frac{d-1}{n-6} & d-\frac{d-1}{n-6} \end{array} \right).
\end{equation*}

Let $\rho'(d,n)=\lambda_2(Q')$. By Corollary \ref{cor}, $\lambda_2(G) \ge \lambda_2(Q')=\rho'(d,n)$. Note that $d$ is odd, and that due to the partition
structure, $n\geq 10$. Thus, to show that $\lambda_2(G)\geq \rho(d,n)$ holds for all $d$ and $n$, we will show that $\lambda_2(G)\geq \rho(d,n)$ holds
when $d=3$ and $n\in \{10,12\}$, and that

\begin{equation}
\label{ineq_rho_pp}
\rho'(d,n) \geq \rho(d,n)
\end{equation}

holds for all other values of $d$ and $n$. To verify that $\lambda_2(G)\geq \rho(d,n)$ holds when $d=3$ and $n\in \{10,12\}$, we compute the
second-largest eigenvalues of all possible multigraphs which have these parameters, and compare them to $\rho(3,10)$ and $\rho(3,12)$, respectively; the
enumeration procedure is described in the Appendix. For all other values of $d$ and $n$, we verify (\ref{ineq_rho_pp}) by separating it into the
following cases and using symbolic computation in Mathematica; see the Appendix for details. See also Case 3 below for a more detailed explanation of why
this computation is sufficient to establish the claim.

\begin{enumerate}
\item[a)] $d=3$, $n\geq 14$. Fix $d=3$ and $x=\frac{1689}{600}$. Then, $\text{det}(xI-Q')>0$ and $\text{det}(xI-Q)<0$ hold for all $n\geq 14$.
\item[b)] $d=5$, $n\in \{10,12\}$. Fix $d=5$ and $x=\frac{47}{10}$. Then, $\text{det}(xI-Q')>0$ and $\text{det}(xI-Q)<0$ hold for $n=10$ and $n=12$.
\item[c)] $d=7$, $n=10$. Fix $d=7$ and $x=\frac{333}{50}$. Then, $\text{det}(xI-Q')>0$ and $\text{det}(xI-Q)<0$ hold for $n=10$.
\item[d)] $d=5$, $n\geq 14$; $d=7$, $n\geq 12$; $d\geq 9$, $n\geq 10$. Fix $x=d-\frac{1}{5}-\frac{1}{n-5}$. Then, $\text{det}(xI-Q')>0$ and
    $\text{det}(xI-Q)<0$ hold for all values of $d$ and $n$ described in this case.
\end{enumerate}

\item[Case 3: $s_1+1\ge 7$.]
In this case, we consider the vertex partition of $G$ with the sets $S_1\cup\{v_1\}$ and $S_2\cup\{v_2\}$; see Figure \ref{case_n7} for an illustration.

\begin{figure}[ht!]
\begin{center}
\includegraphics[scale=0.35]{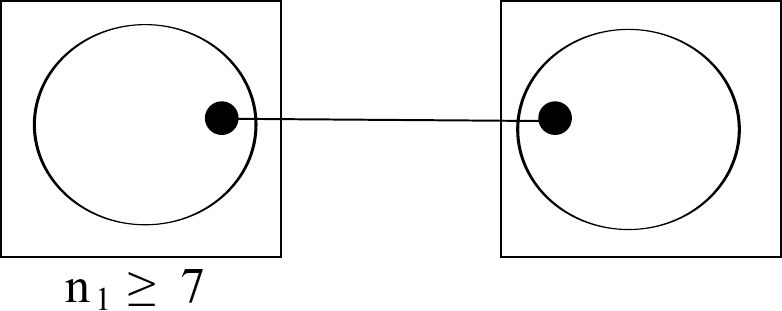}\qquad\qquad
\caption{Partition of $V(G)$ into $S_1\cup\{v_1\}$ and $S_2\cup\{v_2\}$.}
\label{case_n7}
\end{center}
\end{figure}

The second-largest eigenvalue of the quotient matrix $Q''$ corresponding to this vertex partition is equal to $d-\frac 1{s_1+1} - \frac 1{s_2+1}$.
By Corollary \ref{cor},
$\lambda_2(G) \ge \lambda_2(Q'') =  d - \frac 1{s_1+1} - \frac 1{s_2+1} \ge d - \frac 1{7} - \frac 1{n-7}$, where the last inequality follows from the
fact that $s_2+1\geq s_1+1\geq 7$. Note that $n$ is even, $d$ is odd, $d\geq 3$, and due to the partition structure, $n\geq 14$. Thus, to show that
$\lambda_2(G)\geq \rho(d,n)$ holds for all $d$ and $n$, we will show that $\lambda_2(G)\geq \rho(d,n)$ holds when $d=3$ and $n\in \{14,16,18\}$, and
that

\begin{equation}
\label{ineq_rho_7}
d - \frac 1{7} - \frac 1{n-7} \geq \rho(d,n)
\end{equation}

holds for all other values of $d$ and $n$. To verify that $\lambda_2(G)\geq \rho(d,n)$ holds when $d=3$ and $n\in \{14,16,18\}$, we compute the
second-largest eigenvalues of all possible multigraphs which have these parameters, and compare them to $\rho(3,14)$, $\rho(3,16)$, and $\rho(3,16)$,
respectively; the enumeration procedure is described in the Appendix. For all other values of $d$ and $n$, we verify (\ref{ineq_rho_7}) as follows.

Note that $\det(xI-Q)$ is a monic polynomial of degree 4, with roots $\lambda_1(Q)$, $\lambda_2(Q)$, $\lambda_3(Q)$, and $\lambda_4(Q)$; all roots are
real, since they interlace the eigenvalues of $G$. Moreover, $\lambda_1(Q)+\lambda_2(Q)+\lambda_3(Q)+
\lambda_4(Q)=\text{trace}(Q)=d+\frac{d+1}{2}-\frac{d-1}{n-4}$ and $\lambda_1(Q)=d$, which implies that $\lambda_2(Q)+\lambda_3(Q)+
\lambda_4(Q)=\frac{d+1}{2}-\frac{d-1}{n-4}$. By Theorem \ref{theorem4}, $\lambda_2(Q)\geq d-\frac{1}{3}-\frac{1}{n-3}$; thus,
$\lambda_3(Q)+\lambda_4(Q)<0$ for all $d\geq 3$ and $n\geq 14$. Since $\lambda_4(Q)\leq \lambda_3(Q)$, it follows that $\lambda_4(Q)<0$. Finally, note
that
\begin{equation*}
\lambda_4(Q)<0<d-\frac{1}{3}-\frac{1}{n-3}<d-\frac{1}{7}-\frac{1}{n-7}<d=\lambda_1(Q).
\end{equation*}
Thus, showing that (\ref{ineq_rho_7}) holds is equivalent to showing that
\begin{itemize}
\item[a)] $\text{det}(xI-Q)>0$ for $x=d-\frac{1}{3}-\frac{1}{n-3}$, and
\item[b)] $\text{det}(xI-Q)<0$ for $x=d-\frac{1}{7}-\frac{1}{n-7}$,
\end{itemize}
whence it follows that $\lambda_3(Q)\leq d-\frac{1}{3}-\frac{1}{n-3}\leq \lambda_2(Q)\leq d-\frac{1}{7}-\frac{1}{n-7}\leq \lambda_1(Q)$. Using symbolic
computation in Mathematica, we can verify that a) holds for all $d\geq 3$ and $n\geq 14$, and b) holds when $d=3$ and $n\geq 20$, and when $d\geq 5$ and
$n\geq 14$; for details, see the Appendix. Since the case $d=3$, $n\in\{14,16,18\}$ was verified by enumeration, this completes the proof. \qed
\end{description}

\begin{obs}
\label{obs_improvement}
When $t=1$, Theorem \ref{theorem4} states that if $\lambda_2(G)<d-\frac{1}{3}-\frac{1}{n-3}$, then $\kappa'(G)\geq 2$. Case 3 of the proof of Theorem
\ref{theo31} guarantees that $\rho(d,n)>d-1/3-1/(n-3)$, which means that $\rho(d,n)$ is a better bound than $d-1/3-1/(n-3)$.
\end{obs}

\begin{obs}\label{theo4.2tight}
Let $G$ be the $d$-regular multigraph on 6 vertices with $d\geq 3$ and $\kappa'(G)=1$.
Then $\lambda_2(G)=\frac{1}{4}(d-1+\sqrt{9d^2-10d+17})=\rho(d,6)$, where $\rho(d,n)$ is defined as in the statement of Theorem \ref{theo31}. Thus, the
bound in Theorem \ref{theo31} is the best possible for this infinite family of multigraphs.
\end{obs}

\begin{obs}\label{theo4.2rho}
The function $\rho(d,n)$ in Theorem \ref{theo31} behaves like $d-\frac{1}{3}-\frac{1}{n-3}$ as $d$ and $n$ increase.
\end{obs}

\section{Conclusion}

In this paper, we presented two tight upper bounds (Theorems \ref{theo23} and \ref{theo31}) for the second-largest eigenvalues of regular graphs and multigraphs of a given order, which guarantee a desired vertex- or edge-connectivity. The given bounds extend known results for simple graphs, and improve previous results for multigraphs (Theorem \ref{Suilthm1.9} in \cite{O- alg conn mult graphs}). It was also shown that both bounds hold with equality for infinite families of graphs. In deriving these bounds, we used computer-aided symbolic algebra, which
synergizes well with the technique of eigenvalue interlacing; this combination gives a viable approach to investigating spectral bounds guaranteeing graph
theoretic properties, which differs from the typical analytic strategies used in similar results.


In future work, we will aim to extend Theorems \ref{theo23} and \ref{theo31} for all values of $t$. Another problem of interest is to obtain bounds on the second-largest eigenvalues of a graph which guarantee a desired connectivity, and depend on other graph invariants such as girth or circuit rank.

\section*{Acknowledgements}

The authors would like to thank Sebastian Cioab\u{a}, Matthew McGinnis, and an anonymous referee for their helpful suggestions. The authors gratefully acknowledge financial support for this research from the following grants and organizations: NSF-DMS Grants 1604458, 1604773,
1604697 and 1603823 (all authors), The Combinatorics Foundation (A. Abiad, S. O), NSF 1450681 (B. Brimkov), Institute of Mathematics and its Applications
(X. Mart\'{i}nez-Rivera), NRF-2017R1D1A1B03031758 (S. O).


\section*{Appendix}\label{appendix}

Below we provide the Mathematica code used to calculate the minimum of the second root of the characteristic polynomial in the proof of
Theorem~\ref{theo23}, and to check some cases in the proof of Theorem \ref{theo31}. The version of Mathematica used is 10.0.0.0 for 64-bit Microsoft
Windows. We also include additional details about the procedure of enumerating certain multigraphs in the proof of Theorem \ref{theo31}.

\subsection*{Theorem \ref{theo23}: symbolic reductions}
\footnotesize
\begin{onehalfspace}
\noindent\(\pmb{\text{secondroot}=(1/2)(d-\text{m1}/\text{s1}-\text{m2}/\text{s2}+((d-\text{m1}/\text{s1}-\text{m2}/\text{s2}){}^{\wedge}2}\\
\pmb{\hspace{55pt}+4(\text{m1}{}^{\wedge}2/\text{s1}+\text{m2}{}^{\wedge}2/\text{s2}-(\text{m1} \text{m2})/(\text{s1} \text{s2}))){}^{\wedge}(1/2));}\\
\pmb{\text{s2}=n-1-\text{s1};}\\
\pmb{\text{m1}=d-\text{m2};}\\
\pmb{\text{Reduce}[\text{D}[\text{secondroot},\text{m2}]==0\&\&2\leq \text{s1}\leq n-3\&\&\text{m2}>0\&\&d\geq 3,\text{m2}]}\)
\end{onehalfspace}
\begin{doublespace}

\noindent\(\text{s1}\geq 2\&\&n\geq 3+\text{s1}\&\&d\geq 3\&\&\text{m2}==\frac{-d+d n-3 d \text{s1}+2 d n \text{s1}-2 d \text{s1}^2}{-1+n-4 \text{s1}+4
n \text{s1}-4 \text{s1}^2}\)

\noindent\(\pmb{\text{m2}=\frac{-d+d n-3 d \text{s1}+2 d n \text{s1}-2 d \text{s1}^2}{-1+n-4 \text{s1}+4 n \text{s1}-4 \text{s1}^2};}\\
\pmb{\text{FullSimplify}[\text{secondroot}\&\&d\geq 3]}\)

\noindent\(d-\frac{d n}{n+4 n \text{s1}-(1+2 \text{s1})^2}\&\&d\geq 3\)

\noindent\(\pmb{\text{Minimize}\left[\left\{d-\frac{d n}{n+4 n \text{s1}-(1+2 \text{s1})^2},n\geq 5,d\geq 3,2\leq \text{s1}\leq
n-3\right\},\text{s1}\right]}\)

\(\begin{array}{lcr}
\hspace{-20pt}
\begin{cases}
 \frac{-25 d+8 d n}{-25+9 n} & d\geq 3\&\&n\geq 5 \\
 \infty  & \text{True} \\
\end{cases}\,,
&
s1\to\begin{cases}
 \text{Indeterminate} & !(d\geq 3\&\&n\geq 5) \\
 \frac{1}{2} \left(-1-\sqrt{(-5+n)^2}+n\right) & \text{True} \\
\end{cases}
\end{array}
\)
\end{doublespace}

\normalsize
\noindent Note that since $n\geq 5$, the argmin of $s_1$ in the last output is equal to 2.

\subsection*{Theorem \ref{theo31}: symbolic reductions}
\footnotesize

\begin{onehalfspace}

\noindent\(\pmb{Q=\{\{(d+1)/2,(d-1)/2,0,0\},\{d-1,0,1,0\},\{0,1,0,d-1\},}\\
\pmb{\hspace{22pt}\{0,0,(d-1)/(n-4),d-(d-1)/(n-4)\}\};}\\
\pmb{\text{Qprim}=\{\{d-(d-1)/4,(d-1)/4,0,0\},\{d-1,0,1,0\},}\\
\pmb{\hspace{43pt}\{0,1,0,d-1\},\{0,0,(d-1)/(n-6),d-(d-1)/(n-6)\}\};}\\
\pmb{\text{polyQ}=\text{CharacteristicPolynomial}[Q,x];}\\
\pmb{\text{polyQprim}=\text{CharacteristicPolynomial}[\text{Qprim},x];}\)

\normalsize
\noindent Case 2a: note that both inequalities hold for $d=3$ and $n\geq 14$.\\
\footnotesize
\noindent\(\pmb{d=3;x=1689/600;}\\
\pmb{\text{Reduce}[\text{polyQprim}>0\&\&n\geq 10,\text{Integers}]}\\
\pmb{\text{Reduce}[\text{polyQ}<0\&\&n\geq 10,\text{Integers}]}\)

\noindent\(n\in \text{Integers}\&\&n\geq 13\)

\noindent\(n\in \text{Integers}\&\&n\geq 10\)

\normalsize
\noindent Case 2b: note that both inequalities hold for $d=5$ and $n\in \{10,12\}$.\\
\footnotesize
\noindent\(\pmb{d=5;x=47/10;}\\
\pmb{\text{Reduce}[\text{polyQprim}>0\&\&n\geq 10,\text{Integers}]}\\
\pmb{\text{Reduce}[\text{polyQ}<0\&\&n\geq 10,\text{Integers}]}\)

\noindent\(n\in \text{Integers}\&\&n\geq 10\)

\noindent\(n==10\|n==11\|n==12\|n==13\|n==14\|n==15\|n==16\|n==17\)

\normalsize
\noindent Case 2c: note that both inequalities hold for $d=7$ and $n=10$.\\
\footnotesize
\noindent\(\pmb{d=7;x=\frac{333}{50};}\\
\pmb{\text{Reduce}[\text{polyQprim}>0\&\&n\geq 10,\text{Integers}]}\\
\pmb{\text{Reduce}[\text{polyQ}<0\&\&n\geq 10,\text{Integers}]}\)

\noindent\(n\in \text{Integers}\&\&n\geq 10\)

\noindent\(n==10\|n==11\|n==12\|n==13\|n==14\|n==15\)

\normalsize
\noindent Case 2d: note that both inequalities hold for $d=5$ and $n\geq 14$, $d=7$ and $n\geq 12$, and $d\geq 9$ and $n\geq 10$.\\
\footnotesize
\noindent\(\pmb{\text{Clear}[d];x=d-1/5-1/(n-5);}\\
\pmb{\text{Reduce}[\text{polyQ}<0\&\&n\geq 10\&\&d\geq 3,\text{Integers}]}\\
\pmb{\text{Reduce}[\text{polyQprim}>0\&\&n\geq 10\&\&d\geq 3,\text{Integers}]}\)

\noindent\((d|n)\in \text{Integers}\&\&((d==4\&\&n\geq 21)\|(d==5\&\&n\geq 14)\|(d==6\&\&n\geq 12)\|\\(d==7\&\&n\geq 11)\|(d\geq 8\&\&n\geq 10))\)

\noindent\((d|n)\in \text{Integers}\&\&n\geq 10\&\&d\geq 3\)

\normalsize
\noindent Case 3: note that both inequalities hold for $d=3$ and $n\geq 20$, and $d\geq 5$ and $n\geq 14$. The case $d=3$, $n\in\{14,16,18\}$ is verified
by enumeration in the next section.\\
\footnotesize
\noindent\(\pmb{x=d-1/7-1/(n-7);}\\
\pmb{\text{Reduce}[\text{polyQ}<0\&\&n\geq 14\&\&d\geq 3,\text{Integers}]}\\
\pmb{\text{Clear}[x];x=d-1/3-1/(n-3);}\\
\pmb{\text{Reduce}[\text{polyQ}>0\&\&n\geq 14\&\&d\geq 3,\text{Integers}]}\)

\noindent\((d|n)\in \text{Integers}\&\&((d==3\&\&n\geq 19)\|(d\geq 4\&\&n\geq 14))\)

\noindent\((d|n)\in \text{Integers}\&\&n\geq 14\&\&d\geq 3\)

\end{onehalfspace}

\normalsize
\subsection*{Theorem \ref{theo31}: enumerating multigraphs}
Let $A_{10}$ and $A_{12}$ respectively be the sets of 3-regular multigraphs of order $10$ and $12$ with edge-connectivity 1, such that the removal of any
cut-edge of these graphs produces components of order at least 5. Let $A_{14}$, $A_{16}$, and $A_{18}$ respectively be the sets of 3-regular multigraphs of
order $14$, $16$, and $18$ with edge-connectivity 1, such that the removal of any cut-edge of these graphs produces components of order at least 7. These
constraints imply that a graph in $A_{10}$ or $A_{14}$ must have exactly one cut-edge, a graph in $A_{12}$ or $A_{16}$ can have one or two cut-edges, and a
graph in $A_{18}$ can have one, two, or three cut-edges.

For $i\in\{5,7,9,11\}$, let $B_i$ be the set of all connected multigraphs which have degree sequence $\{3^{i-1},2\}$ and have no cut-edges. For any graph
$H\in B_i$, $i\in\{5,7,9,11\}$, define $v_2(H)$ to be the degree 2 vertex of $H$. Let $J_2$ be the graph consisting of two vertices joined by a double
edge, let $J_4$ be the graph obtained by joining two copies of $J_2$ by one edge, and let $J_4'$ be a complete graph on four vertices with one edge
removed. For $J\in\{J_2,J_4,J_4'\}$, define $v_2(J)$ to be one of the degree 2 vertices of $J$, and $v_2'(J)$ to be the other degree 2 vertex of $J$.
For any $i,j\in\{5,7,9,11\}$, define $B_i\leftrightharpoons B_j$ to be the set $\{H\dot\cup H'+\{v_2(H),v_2(H')\}:H\in B_i,H'\in B_j\}$ (where $\dot\cup$
denotes disjoint union). For any $i,j\in\{5,7,9,11\}$ and $J\in\{J_2,J_4,J_4'\}$, define $B_i\leftrightharpoons J\leftrightharpoons B_j$ to be the set
$\{H\dot\cup H'\dot\cup J+\{\{v_2(H),v_2(J)\},\{v_2(H'),v_2'(J)\}\}:H\in B_i,H'\in B_j\}$. In other words, ``$\leftrightharpoons$'' denotes the set
obtained by joining all possible pairs of graphs from the indicated families by a cut-edge incident to their degree 2 vertices. With this in mind, it is
easy to see that
\begin{eqnarray*}
A_{10}&=&B_5\leftrightharpoons B_5\\
A_{12}&=&(B_5\leftrightharpoons B_7)\cup (B_5\leftrightharpoons J_2\leftrightharpoons B_5)\\
A_{14}&=&B_7\leftrightharpoons B_7\\
A_{16}&=&(B_7\leftrightharpoons B_9)\cup (B_7\leftrightharpoons J_2\leftrightharpoons B_7)\\
A_{18}&=&(B_7\leftrightharpoons B_{11})\cup (B_9\leftrightharpoons B_9)\cup (B_7\leftrightharpoons J_2\leftrightharpoons B_9)\cup \\
&&(B_7\leftrightharpoons J_4\leftrightharpoons B_7)\cup (B_7\leftrightharpoons J_4'\leftrightharpoons B_7).
\end{eqnarray*}

\noindent See Figure \ref{enumeration_struct} for an illustration of these constructions.

\begin{figure}[ht!]
\begin{center}
\includegraphics[scale=0.45]{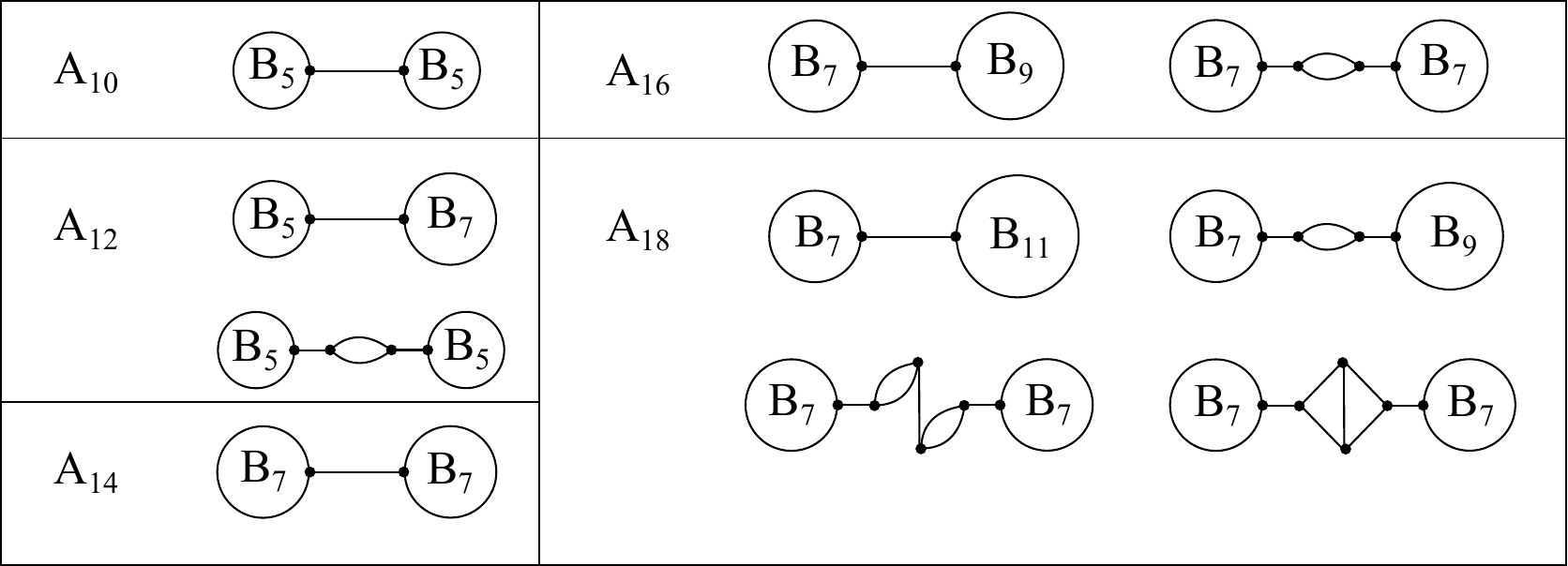}\qquad\qquad
\caption{All possible 2-vertex-connected component and cut-edge structures of graphs in $A_i,i\in\{10,12,14,16,18\}$.}
\label{enumeration_struct}
\end{center}
\end{figure}

Thus, to find the graphs in $A_i$, $i\in\{10,12,14,16,18\}$, it suffices to find the graphs in $B_j$, $j\in\{5,7,9,11\}$. Since the graphs in $B_j$ are
3-regular and connected, they cannot have triple edges; moreover, they can have at most $\frac{j-1}{2}$ double edges. Let $M(\ell,j)$ be the set of
multigraphs in $B_j$ which have $\ell$ double edges. Then, $B_j=M(0,j)\cup\cdots\cup M(\frac{j-1}{2},j)$. We will now describe a procedure for enumerating
the graphs in $M(\ell,j)$.

If the double edges of the graphs in $M(\ell,j)$ are replaced by single edges, the resulting graphs will be simple, 2-vertex-connected, and have degree
sequence $\{3^{j-2\ell-1},2^{2\ell+1}\}$. There are well-known algorithms for generating all nonisomorphic simple graphs with a given degree sequence (cf.
\cite{hakimi,hanlon,ruskey}); a practical algorithm is implemented in the software system SageMath. Let $S(\ell,j)$ be the set of nonisomorphic simple
graphs with degree sequence $\{3^{j-2\ell-1},2^{2\ell+1}\}$. Then, by adding double edges in all feasible ways to the simple graphs in $S(\ell,j)$, we can
recover the multigraphs in $M(\ell,j)$. Specifically, a double edge can be added to a graph in $S(\ell,j)$ only where a single edge with two degree 2
endpoints already exists. Moreover, not every graph in $S(\ell,j)$ can have $\ell$ double edges added to it in a way that the resulting multigraph is in
$M(\ell,j)$; similarly, it may be possible to add $\ell$ double edges to a graph in $S(\ell,j)$ in multiple ways so that the resulting multigraphs are in
$M(\ell,j)$.

Let $H$ be a graph in $S(\ell,j)$ and let $f(H)$ be the subgraph induced by the degree 2 vertices of $H$. Since the maximum degree of $f(H)$ is 2, $f(H)$
is the disjoint union of some paths and cycles. However, if $f(H)$ contains a cycle with less than $j$ vertices, a multigraph in $M(\ell,j)$ cannot be
obtained by doubling single edges of $H$ with two degree 2 endpoints (since any resulting multigraph with degree sequence $\{3^{j-1},2\}$ will be
disconnected). Similarly, if $f(H)$ contains more than one odd path, a multigraph in $M(\ell,j)$ cannot be obtained by doubling single edges of $H$ with
two degree 2 endpoints (since any resulting multigraph with degree sequence $\{3^{j-1},2\}$ will not have $\ell$ multiple edges).

Thus, let $S'(\ell,j)=\{H\in S(\ell,j):f(H)$ is either a cycle $C_j$, or contains exactly one odd path$\}$. For any graph $H$ in $S'(\ell,j)$, the
different maximum matchings (i.e. $\ell$-matchings) of $f(H)$ correspond to different ways to add double edges to $H$. Let $F(H)$ be the set of multigraphs
obtained by adding double edges to $H$ corresponding to the different $\ell$-matchings of $f(H)$. Then, $M(\ell,j)=\bigcup_{H\in S'(\ell,j)}F(H)$,
$B_j=\bigcup_{\ell=0}^{(j-1)/2} M(\ell,j)$, and $A_i$ can be obtained by joining pairs of graphs in $B_j$ as described earlier. Note that the set of
distinct maximum matchings of a graph whose components are paths, one of which is odd, can be found in linear time. In particular, in the even paths, there is a single way to maximally match up the edges; in the odd path of length $p$, there are $(p+1)/2$ different ways to match up the edges (and some of them may lead to isomorphic graphs, which can be tested for or ignored).

See Figure \ref{enumeration2} for an illustration of this enumeration for $M(2,7)$; the other sets of multigraphs $M(\ell,j)$ are handled analogously, and
combined to obtain the graphs in $A_i$. Finally, for each multigraph in $A_i$, we can easily compute and compare the second-largest eigenvalue to
$\rho(3,i)$; we have found that all of these eigenvalues are greater than or equal to $\rho(3,i)$, as desired.

\begin{figure}[ht!]
\begin{center}
\includegraphics[scale=0.2]{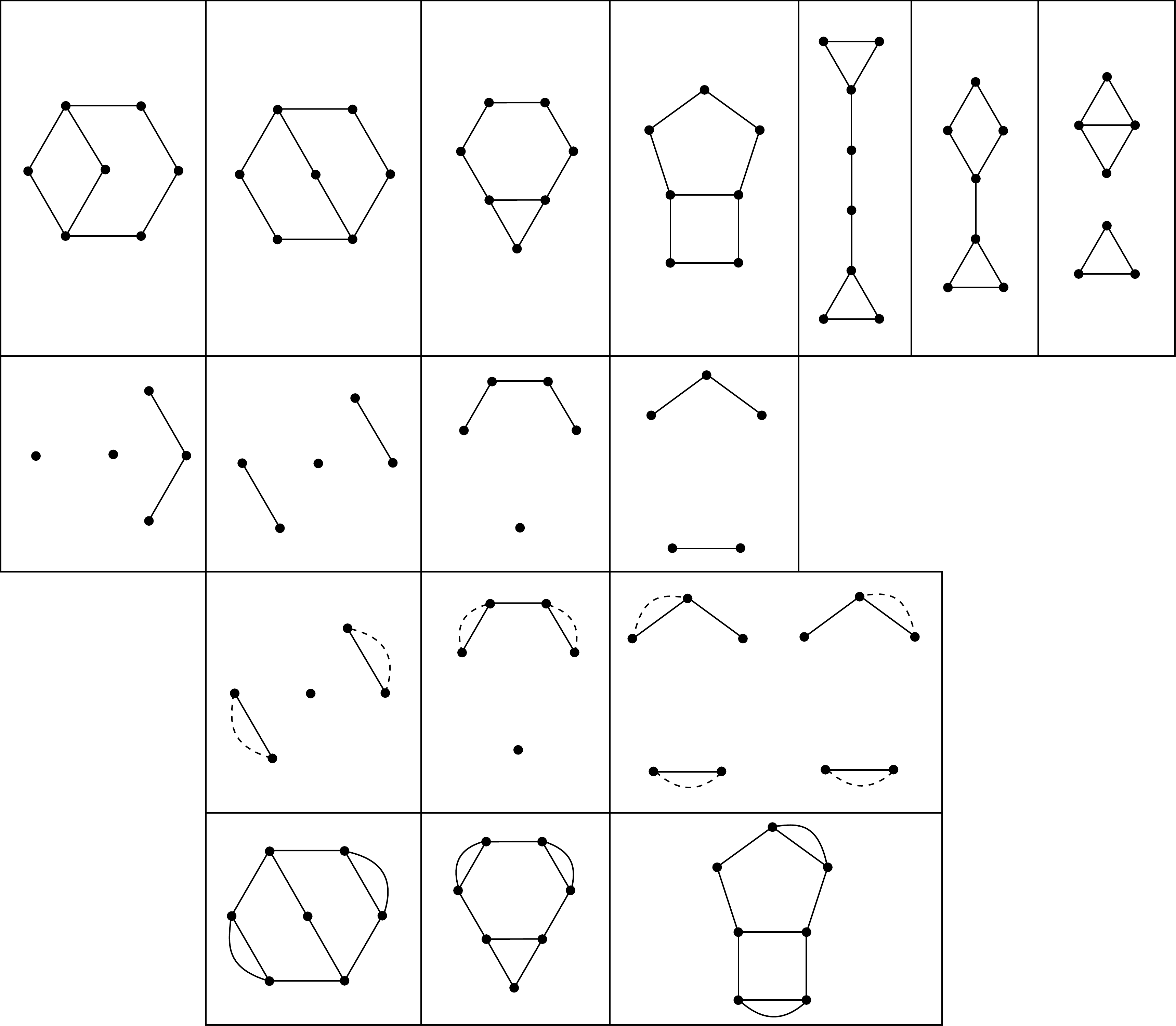}
\caption{Enumerating the graphs in $M(2,7)$. \emph{Top row}: the graphs in $S(2,7)$; the three graphs on the right are not 2-vertex-connected, so they are
not considered further. \emph{Second row}: $f(H)$ for the remaining graphs $H$; the graph on the left has multiple odd paths, so it is not considered
further. \emph{Third row}: all possible 2-matchings of the remaining graphs in the second row. \emph{Bottom row}: adding double edges specified by the
matchings to obtain the graphs in $M(2,7)$; the two matchings of the graph on the right happen to result in isomorphic multigraphs.}
\label{enumeration2}
\end{center}
\end{figure}

\end{document}